\documentclass{amsart}
\usepackage{amsmath}
\usepackage{graphicx}

\newtheorem{theorem}{Theorem}[section]
\newtheorem{cor}[theorem]{Corollary}

\newtheorem{ques}[theorem]{Question}

\newtheorem{prp}[theorem]{Proposition}

\newcommand{\R}{\mathbb{R}}

\newcommand{\PP}{\mathcal{P}}

\title{Most graphs are knotted}
\author{Kazuhiro Ichihara}
\address{Department of Mathematics, College of Humanities and Sciences, Nihon University, 3-25-40 Sakurajosui, Setagaya-ku, Tokyo 156-8550, Japan.}
\email{ichihara.kazuhiro@nihon-u.ac.jp}
\thanks{Ichihara is partially supported by Grant-in-Aids for Scientific Research (C) (No. 18K03287), The Ministry of Education, Culture, Sports, Science and Technology, Japan, respectively. }

\author{Thomas W.\ Mattman}
\address{Department of Mathematics and Statistics,
California State University, Chico,
Chico, CA 95929-0525}
\email{TMattman@CSUChico.edu}

\begin{document}

\begin{abstract}
We present four models for a random graph and show that, in each case, 
the probability that a graph is intrinsically knotted goes to one as the number of vertices
increases. We also argue that, for $k \geq 18$, most graphs of order $k$
are intrinsically knotted and, for $k \geq 2n+9$, most of order $k$ are not
$n$-apex. 
We observe that $p(n) = 1/n$ is the threshold for intrinsic knotting and linking
in Gilbert's model.
\end{abstract}

\maketitle

\section*{Introduction}
We investigate knotting of random graphs. A graph is {\em intrinsically knotted} (IK), 
if every tame embedding in $\R^3$ contains a non-trivially knotted cycle. For us, graphs are simple (no 
loops or multiple edges) and we identify the combinatorial object with the associated 
$1$-dimensional cell complex.

Below, we list four models for a random graph.
Models 1 and 2 are well-known in graph theory, see \cite{B1} or \cite[Chapter VII]{B2}.
Let $|G| = n$ denote the {\em order} or number of vertices of a graph. Let $N = \binom{n}{2}$
denote the number of edges in $K_n$, the complete graph of order $n$.

\begin{description}
\item[1]
(Erd\H{o}s-R\'enyi~\cite{ER}) Choose a graph $G(n,M)$ uniformly at random from the set of labelled graphs 
with $n$ vertices and $M$ edges. There are $\binom{N}{M}$ such graphs and the probability of 
choosing a particular graph is ${\binom{N}{M}}^{-1}$.
\item[2]
(Gilbert~\cite{G}) For each of the possible $N$ edges, we select it as an edge of the graph $G(n,p)$ independently 
with probability $p$. 
\item[2.5]
If $p = \frac12$ in Gilbert's model, then every one of the $2^N$ labelled graphs on $n$
vertices is equally likely. The probability of choosing a particular labelled graph with $|G| = n$ is then $2^{-N}$.
\item[3]
(Unlabelled version of Model 2.5) 
Let $\Gamma_n$ denote the number of unlabelled graphs on $n$ vertices. Choose a graph
from this set uniformly at random. The probability of choosing a particular unlabelled graph with $|G| = n$ is
$\Gamma_n^{-1}$.
\end{description}

There are two senses to our title. 
In Section~3, we argue that, in all four models, the probability that a graph is intrinsically knotted goes to one as the
number of vertices increases.

In Section~1, we show that there is a constant $n_{IK}$ 
such that, when $n \geq n_{IK}$, most order $n$ graphs in Model 2.5 or 3 are intrinsically knotted. 
We show that $13 \leq n_{IK} \leq 18$, but leave open the question of the exact value of $n_{IK}$.

Section~2 develops the idea of $n$-apex. A graph is $n$-apex if it becomes planar on the deletion
of $n$ or fewer vertices. We'll explain how this relates to intrinsic knotting as well as intrinsic linking; 
a graph is {\em intrinsically linked} (IL) if every tame embedding in $\R^3$ includes a pair of 
non-trivially linked cycles. We conclude the section by showing that when $k \geq 2n+9$, 
most order $k$ graphs in Model 2.5 or 3 are not $n$-apex.

Much of our paper is based on a fundamental observation of Mader~\cite{M}
(see also \cite{CMOPRW}). Let $\|G\|$ denote the {\em size} or number of edges of graph $G$.
Mader showed that if $|G| = n \geq 7$ and $\|G\| \geq 5n-14$, then $G$ has a $K_7$ minor.
Since $K_7$ is IK~\cite{CG}, any graph with a $K_7$ minor is IK.
Thus, a graph of order seven or more with at least $5|G|-14$ edges  is IK.

As a consequence of Robertson, Seymour, and Thomas's~\cite{RST} classification of obstructions for 
linkless embedding, an intrinsically knotted graph is also intrinsically linked. Thus, our main conclusions
about knotting of random graphs carry through to linking. For example, an alternate title for our paper is
`Most Graphs are Knotted and Linked.'

We remark that, in Gilbert's $G(n,p)$ model (Model 2), 
 $p(n) = 1/n$ is the threshold for intrinsic linking and knotting, just as it is for nonplanarity.
This is because Ajtai, Koml\'os, and Szemer\'edi~\cite{AKS} showed that, for any fixed $r$, almost all 
graphs contain a topological $K_r$ when $p(n) = c/n$ for $c > 1$. Nonplanarity, intrinsic linking, 
and intrinsic knotting are guaranteed by a topological $K_r$ with $r = 5,6,7$, respectively.

\section{The complement bound for intrinsic knotting} 

In this section we define a constant $n_{IK}$, which is at most 18, and show that, in Models 2.5 and 3, 
if $n \geq n_{IK}$, most graphs of order $n$ are IK. 

Suppose $\PP$ is a graph property whose negation is closed under taking minors and 
such that $K_n$ has $\PP$ for sufficiently large $n$.
Examples include intrinsic knotting and linking.
Let $n_{\PP}$, the {\em complement bound for $\PP$}, denote the least $n$ such that, 
for every $G$ with $|G| \geq n$, either $G$ or its complement has $\PP$. For example, the complement
bound for nonplanar is $n_{NP} = 9$~\cite{BHK}.

\begin{prp} In Model 2.5 or 3, if $n \geq n_{\PP}$, then most graphs with $|G| = n$ have $\PP$.
\end{prp}

\begin{proof}
Let $n \geq n_p$. Pair off each order $n$ graph $G$ with its complement $\overline{G}$. Since $|G| \geq n_p$
at least one of these two has $\PP$. If $G$ is self complementary, then $G$ must have $\PP$. 
This means at least half the graphs have $\PP$.
\end{proof}

Below, we estimate $n_{IK}$, the complement bound for intrinsic knotting. 
However, we leave open the following question.

\begin{ques} What is the complement bound for intrinsic knotting, i.e., the least $n$ such that,
if $|G| \geq n$, either $G$ or its complement is intrinsically knotted?
\end{ques}

We observe that $13 \leq n_{IK} \leq 18$.
For the lower bound, use the 12 vertex 
self complementary and non IK graph described in \cite{PP}.
On the other hand, since $\binom{18}{2} = 153$,
if $|G| = 18$, then either $G$ or its complement has at least 77 edges. But, since $5(18)-14 = 76$,
by Mader's result, that graph is IK.

Recently, the Pavelescu's~\cite{PP2} investigated the complement bound for intrinsic linking
and showed $11 \leq n_{IL} \leq 13$.

\begin{cor} In Model 2.5 or 3, if $n \geq 18$, then most graphs of order $n$ are intrinsically knotted.
If $n \geq 13$, then most graphs of order $n$ are intrinsically linked.
\end{cor}

\section{The $n$-apex property}

In this section we relate the $n$-apex property to intrinsic knotting and linking and 
show that the complement bound for the property not $n$-apex is at most $2n+9$.

We may view intrinsic knotting and linking as generalizations of the nonplanar property due to the connection
with complete minors. As mentioned in the introduction, nonplanarity, intrinsic linking, 
and intrinsic knotting are guaranteed by a $K_r$ minor with $r = 5,6,7$, respectively.
Thus, our study of $n_{IK}$, and the Pavelescu's of $n_{IL}$, can be seen as a generalization of
the proof of Battle, Harary, and Kodama~\cite{BHK} that $n_{NP} = 9$.

The $n$-apex property is a wider ranging generalization along the same lines: a graph with 
a $K_{n+5}$ minor is not $n$-apex. We can even equate planarity with $0$-apex.
Below we parlay the observation that $n_{NP} = 9$ into a natural generalization, $n_{NnA} \leq 2n+9$.

This suggests that the $1$-apex property, which is also called, simply, apex, is related to IL and 
$2$-apex to IK. Indeed, Sachs~\cite{S} observed that an IL graph is not apex and two groups~\cite{BBFFHL, OT} independently
established that an IK graph is not $2$-apex. Earlier we used the order 12 self complementary graph of \cite{PP} to see that
$n_{IK} \geq 13$. This graph is not IK because it is $2$-apex. The same graph, therefore, also establishes that $n_{N2A} \geq 13$. 
Similarly, to argue that $n_{IL} \geq 11$, the Pavelescu's employed an order 10 graph $G$ such that both $G$ and its complement
are apex. Thus, this same graph also shows $n_{NA} \geq 11$, where $n_{NA}$ is the complement bound for not apex. Combined
with the following proposition, this shows $n_{NA} = 11$ and $n_{N2A} = 13$, which makes for a nice continuation of
the planar, or $0$-apex, value, $n_{NP} = 9$.

\begin{prp} 
Let $n_{NnA}$ denote the complement bound for not $n$-apex. Then $n_{NnA} \leq 2n+9$.
\end{prp}

\begin{proof}
Let $|G| \geq 2n+9$. We will argue $G$ or its complement is not $n$-apex. For a contradiction, suppose both $G$ and 
$\overline{G}$ are $n$-apex. Let $A \subset V(G)$ be an apex set for $G$ and $B$ an apex set for $\overline{G}$.
That is, the induced graph on $V(G) \setminus A$ is planar and similarly for $B$. Let $H$ be the induced subgraph on 
$V(G) \setminus (A \cup B)$. As both $G$ and $\overline{G}$ are $n$-apex, $|H| \geq 9$ and both $H$ and $\overline{H}$
are planar. This contradicts $n_{NP} = 9$.
\end{proof}

\begin{cor} In Model 2.5 or 3, if $k \geq 2n+9$, then most graphs of order $k$ are not $n$-apex.
\end{cor}

In light of the corollary, we can also say, `Most graphs are not $n$-apex.'

\section{A graph is knotted with probability one.}

In this section, we show that the probability that a graph is intrinsically knotted goes to one with increasing graph order.

Assume $0<p \leq 1$ in Gilbert's model, Model 2, and let
$|G| = n \geq 7$. Using Mader's observation, $G$ is IK if $\|G\| \geq 5n-14$. So the probability that a graph 
is not IK is bounded by the probability that it has at most $5n-15$ edges:
$$\mbox{Prob} (G \mbox{ not IK}) \leq \mbox{Prob} (\|G \| \leq 5n-15) =
 \sum_{k=0}^{5n-15} \binom{N}{k}p^k(1-p)^{N-k} \leq e^{-2t^2N}.$$
The last inequality is due to Hoeffding~\cite{H}, with  $t = p - (5n-15)/N$,
and shows that the probability approaches zero as $n$ goes to infinity.

In particular, the probability a graph is not IK goes to $0$ with $n$ in 
Model 2.5, when $p = \frac12$. Model 1 behaves similarly if we assume that
there is a positive proportion of edges, $M/N = p > 0$; see~\cite[Chapter VII, Theorem 6]{B2}, for example.

For Model 3, assume $|G| = n \geq 7$ and let $r = 5n-15$. By Mader's result, a graph that is not IK
has $\|G\| \leq r$. Let $q = \lfloor N/2 \rfloor$ and assume further
that $n \geq 18$ so that $r < q$.
The number of unlabelled graphs with $\|G\| = k$ 
is bounded above by the corresponding number 
of labelled graphs, $\binom{N}{k}$. So, the number of 
non IK graphs in Model 3 is bounded above by
$$\sum_{k=0}^{r} \binom{N}{k} < (r+1) \binom{N}{r}.$$

Since $q > r$, the unlabelled graphs of size $q$ are IK. 
On the other hand, Wright~\cite{W} showed that the number of such graphs is asymptotic
to $\binom{N}{q}/n!$. Therefore, as the order, $n$, goes to infinity, the proportion of graphs in Model 3 
that are not IK is bounded above by
$$\frac{(r+1) \binom{N}{r}}{\binom{N}{q}/n!} = (r+1)n! \frac{q(q-1)(q-2) \cdots (r+1)}{(N-r)(N-r-1) \cdots (N-q+1)} < (r+1)n! \left( \frac{q}{N-r} \right)^{q-r}.$$

Next, assume $n > 105$ so that $q-r > n^2/5$ and $q/(N-r) < 3/5$. This means,
$$\left( \frac{q}{N-r} \right)^{q-r} < \left( \frac35 \right)^{n^2/5} < \left( \frac{1}{(5/3)^{1/5}} \right)^{n^2}  < \frac{1}{(1.1)^{n^2}}.$$

Thus, in the limit as the number of vertices $n$ goes to infinity, the proportion of non IK graphs is bounded above by

$$\frac{ (r+1) n! }{ ( (1.1)^n )^{n} } < \frac{r+1}{n} \cdot \left( \frac{ n }{ (1.1)^n } \right)^n, $$
which goes to 0, as required.

\end{document}